\documentclass[english]{article}
\usepackage{comment} 
\usepackage[utf8]{inputenc}
\usepackage[T1]{fontenc}
\usepackage{graphicx} 
\usepackage{amsmath,amssymb,amsfonts}%
\usepackage[margin=3cm]{geometry}
\usepackage{amsthm}
\usepackage{dsfont}
\usepackage{babel}
\usepackage{subcaption}
\usepackage[linesnumbered,ruled,vlined]{algorithm2e}
\usepackage{xcolor}
\usepackage{csquotes}
\usepackage{authblk}
\usepackage{esvect}
\usepackage{mathtools}
\usepackage{xparse}

\title{From entropic transport to martingale transport, and applications to model calibration}

\author[1]{Jean-David Benamou}
\author[1,2]{Guillaume Chazareix}
\author[2]{Grégoire Loeper}

\affil[1]{INRIA Paris}
\affil[2]{BNP Paribas Global Markets}

\NewDocumentCommand{\expect}{ e{_} e{^} s o >{\SplitArgument{1}{|}}m }{%
  \operatorname{\mathbb{E}}
  \IfValueT{#1}{{\!}_{#1}}
  \IfValueT{#2}{{\!}^{#2}}
  \IfBooleanTF{#3}{
    \expectarg*{\expectvar#5}%
  }{
    \IfNoValueTF{#4}{
      \expectarg{\expectvar#5}%
    }{
      \expectarg[#4]{\expectvar#5}%
    }%
  }%
}
\NewDocumentCommand{\expectvar}{mm}{%
  #1\IfValueT{#2}{\nonscript\;\delimsize\vert\nonscript\;#2}%
}
\DeclarePairedDelimiterX{\expectarg}[1]{[}{]}{#1}

\newcommand*\diff{\mathop{}\!\mathrm{d}}

\usepackage[backend=bibtex,style=numeric,doi=false,eprint=false,url=false]{biblatex}
\DefineBibliographyStrings{english}{
  andothers = {et\addabbrvspace al\adddot},
}

\newcommand{\as}{\text{a.s.}}
\newcommand{\Nte}[1]{{#1}  } 
    
\newcommand{\SP}[3]{\langle #1  ,  #2   \rangle_{#3} }

\newcommand{\RR}{{W}}

\NewDocumentCommand{\dint}{e{_} e{^}}{%
  {\displaystyle \int\IfValueT{#1}{_{#1}}\IfValueT{#2}{^{#2}}}
}

\newcommand{\bs}{{\overline{\sigma}}}
\newcommand{\E}{\mathbb{E}}
\newcommand{\F}{\mathcal{F}}
\newcommand{\PP}{\mathbb{P}}

\NewDocumentCommand{\bPP}{e{_} e{^}}{%
  {\overline{\PP\IfValueT{#1}{_{#1}}\IfValueT{#2}{^{#2}}}}
}

\newcommand{\R}{\mathbb{R}}

\newcommand{\Pc}{\mathcal{P}}

\newcommand{\Mc}{\mathcal{M}}
\newcommand{\Gc}{\mathcal{G}}
\newcommand{\Kc}{\mathcal{K}}
\newcommand{\Fc}{\mathcal{F}}
\newcommand{\Ic}{\mathcal{I}}
\newcommand{\Sc}{\mathcal{S}}

\newcommand{\Xc}{\mathcal{X}}

\newcommand{\C}{\mathcal{C}}

\DeclareMathOperator{\KL}{KL}

\renewcommand{\vv}[1]{{#1}}

\newtheorem{proposition}{Proposition}[section]
\newtheorem{definition}{Definition}[section]
\newtheorem{theorem}{Theorem}[section]
\newtheorem{remark}{Remark}[section]
\newtheorem{lemma}{Lemma}[section]

\begin{document}

\maketitle
\begin{abstract}
  We propose a discrete time formulation of the semi-martingale optimal transport problem based on multi-marginal entropic transport.
  This approach offers a new way to formulate and solve numerically
  the calibration problem proposed by \cite{guo2021optimal}, using a
  multi-marginal extension of Sinkhorn algorithm as in \cite{BenamouE,CarlierGFE,BenamouMFG2}.
  When  the time step goes to zero we recover, as detailed in the companion paper \cite{Mpaper},  a continuous semi-martingale process, solution to a semi-martingale optimal transport problem, with a cost function involving the so-called ``specific entropy'',  introduced in
  \cite{Gantert91}, see also \cite{follmer22} and \cite{backhoff23}.

\end{abstract}

\section{Introduction}
Applications of Semi Martingale Optimal Transport (SMOT) in finance have been
the object of several recent studies (\cite{Tan_2013}, \cite{guo2021optimal}, \cite{Guyon_2022} \ldots). This framework is particularly well adapted  to the problem of model calibration: Find a diffusion model that is compatible with observed option prices.
SMOT is the stochastic version of
Dynamic Optimal Transport (DOT), that was  introduced by \cite{BBA}, as a generalization of static OT  where the mass is transported  by a time dependent flow
minimizing the kinetic energy.
While the theoretical aspects of these problems are now well understood, the numerical implementation remains challenging.

In the meantime, a stochastic relaxation of static optimal transport, known as {\it Entropic} Optimal Transport (EOT), has shown to be solvable efficiently, by the so-called {\it Sinkhorn} algorithm (see \cite{PeyreB} for a review). Interestingly,  while there is equivalence between the static OT problem and its dynamic version, the Entropic regularisation of OT can also be seen either as  a static problem, or as variant of the DOT adding a constant volatility diffusion to the governing model, this dynamic problem is known as the Schr\"odinger problem, see \cite{LeonardS}. 

Consider  processes described by the SDE with drift $\mu$ and volatility $\sigma$
\[
    \diff X_t = \mu_t \, \diff t + \sigma_t \, {\diff W}_t, \quad  X_0 \sim \nu_0,
\]
where $W_t$ is the Wiener process and $\nu_0$ the initial law of $X$ (typically ${X_0}_\# \PP = \delta_{x_0}$).  Formally, the above SDE induces a probability $\PP$ on the  space of continuous paths. 
\Nte{Reciprocally, see  section \ref{sec:continuous},  $\mu$ and $\sigma$ can be interpreted as characteristics coefficients depending on 
$\PP$. }In this setting, all problems (DOT, SMOT and EOT) are  seen as a variant of :  
\begin{align}\label{eq:SMOT}
  \inf_{ \displaystyle \PP}  {\cal F}(\PP) \coloneqq \expect_\PP*{\dint_0^T F(t,X_t,\mu_t(X_t),\sigma_t(X_T)^2) \diff t}  +  M_0(\PP_0) 
\end{align}
The cost function $F$ may describe modelling, calibration or observations  constraints or combinations of these,  and $M_0$ enforces the $\nu_0$ initial marginal constraint. 
\Nte{
\begin{itemize} 
\item   Classical DOT  \cite{BBA} corresponds to the particular case $F=|\mu|^2 $ if $\sigma \equiv 0$, ${ X_{t_0}}_\# \PP = \rho_0 $ and 
 $ {X_{T}}_\# \PP = \rho_T $ , $+\infty$ otherwise.  It corresponds to the minimization of the kinetic energy under 
 under initial and terminal conditions on the laws of $X_{t=0}$ and $X_{t=T}$ denoted $\rho_0$ and $\rho_T$.   
\item EOT   \cite{LeonardS}  is the same as DOT except for a prescribed constant volatility $\sigma \equiv \bs $.
It is also one of the formulations of the {\it Schr\"odinger's} problem , i.e. minimizing the relative entropy (aka Kullblack Leibler divergence) of $\PP$:  $\KL(\PP | \RR_\bs)$, with respect to the  Wiener measure $\RR_\bs$ (with a given constant volatility $\bs$)
Classical EOT in its static formulation can be solved efficiently by the Sinkhorn's algorithm  \cite{CuturiA} \cite{SalanieA} \cite{BCCPN}.
\item SMOT  will handle general forms of $F$, as long as it is convex with respect to $(\mu,\sigma^2)$. 
\item  Martingale Optimal Transport (see  \cite{backhoffM} and the references therein) implies a  $\mu \equiv 0$  constraint, a pure diffusion model.
For the  classical martingale constrained DOT problem, feasibility requires  that  $\rho_0$ and $\rho_T$ to be in convex order. In this paper we  use a soft version  of this constraint $F = C \|\mu|^2$ with $ C>>1$.  We therefore try to achieve a mass transport  governed preferably  by the  volatility and not the drift. 
\item Price calibration  adds  a set of  discrete constraints described by $(\tau_i, c_i, G_i)_{i \in \Ic}$ where for each $i$ the  triplet $(\tau_i, c_i, G_i)$ is the maturity, the price and the payoff function of an observed derivative price on the market. 
In this paper  penalize the cost with  $ C_k(X)  =   \lambda \sum_{i \in \Ic}   \| \expect { G_i(X_{\tau_i}) } - c_i  \|^2 $ ($\lambda >>1)$
(see next section).
\end{itemize} }

A general approach to DOT, EOT and their generalisations  (see \cite{benamou2021} for a review) is a time discretization that leads to a  Multi-Marginal OT problem.  In this setting the minimization is performed over the law $\PP^h$ ($h$ is the time step) of a vector-valued random variable whose marginals represent densities at each time step.
In this paper, we  use this time discretization method and an
entropic regulariization  to solve problem \eqref{eq:SMOT}:
We minimize  the sum  of a discretized form of $\F(\PP)$ in
\eqref{eq:SMOT}, $\F_h(\PP^h)$.  We use in particular the natural 
 discrete versions of $\mu$ and $\sigma$ (see section \ref{sec:prelim}) 
and the discrete time relative entropy regularization $\KL(\PP^h | \RR^h_\bs)$.

The drawback of minimising an energy in the form $\KL(\PP | \RR_\bs)$
is that by essence the minimizer is constrained at $\sigma=\bs$ and cannot satisfy constraints on $\mu$ (for instance $\mu\equiv 0$ or $\mu=r$, the interest rate) familiar in finance, since $\mu$ is precisely the only degree of freedom used to comply with  the distribution constraints.
We propose   to overcome this issue by
considering a proper scaling of the discrete relative entropy and its convergence property :
if $\PP^h$ a sequence of Markov chains converging to the law of a diffusion $\PP$ with  drift $\mu$ and  volatility
$\sigma$ then
\begin{align}\label{specrelent}
    \lim_{h\searrow 0 }  h\, \KL(\PP^h | \RR^h_\bs)   = \expect_\PP*{\dint_0^T \frac{\sigma_t^2}{\bs^2} - 1 - \log{\frac{\sigma_t^2}{\bs^2}} \, \diff t} =:  \Sc(\PP |\RR_\bs ),
\end{align}
The "specific relative entropy" $\Sc$ defined above has been introduced in
\cite{Gantert91} see also  \cite{follmer22} \cite{backhoff23}. \\

It is shown in \cite{Mpaper}  that minimizers of  time discrete problem
\begin{equation} 
\label{Ph} 
\inf_{\PP^h}  \F_h(\PP^h) + h\, KL(\PP^h | \RR^h_\bs)   + M_0(\PP^h_0)
\end{equation} 
approximate  in the limit $h\searrow 0$  a diffusion process $\PP$   minimizing  
 \begin{equation} 
\label{P} 
  \inf_{\PP}   {\F}(\PP) +\Sc(\PP | \RR_\bs).
\end{equation} 
For $h >0$ the minimizer $\PP^h$ is a solutions of a  Multi-Marginal EOT problem, and a discrete Markov chain that can be used for simulations.
The interest of this approach is not only theoretical. Classical methods to solve \eqref{eq:SMOT} involve maximizing the dual problem through gradient ascent or primal-dual approaches.
 These methods imply  solving a fully non-linear Hamilton-Jacobi-Bellman equation at each iteration (\cite{BBA}, \cite{PapaA}, \cite{LoeperA}). 
Our approach by Multi-Marginal Sinkhorn's algorithm, extending \cite{BenamouE} and \cite{BenamouMFG2},  computes an equivalent time discrete solution trough 
the maximisation of a strictly concave variational problem.

This paper describes the dual formulation of the problem in the context of local volatility calibration,   the associated Sinkhorn algorithm and its practical implementation, with numerical examples.
\Nte{A numerical study as $h \searrow 0$ is available in \cite{Mpaper}}.

\section{Martingale Optimal Transport for model calibration}
\label{sec:continuous}
The continuous formulation of (Semi-)Martingale Optimal Transport was introduced in \cite{Tan_2013}, and extended for multiple calibration applications as presented in the survey \cite{guo2021optimal}. Let $\Omega = C([0, T], \R), \, T > 0$ be the set of continuous paths, and $\Pc$ the set (or a convex subset of) probability measures on $\Omega$.

Following \cite{Tan_2013}, we 
 formulate the problem as a minimization problem on $\Pc$, we restrain our search to the set
${\cal P}^0_s\subset {\cal P}$ such that, for each $\PP\in{\cal P}^0$, $X\in \Omega$ is an $({\cal F},\PP)$-semimartingale on $[0,1]$ given by
\[
    X_t=X_0+A^\PP_t + M^\PP_t,\quad \langle X\rangle_t=\langle M^\PP_t\rangle=B^\PP, \quad \PP\text{-}\as, \quad t\in[0,1],
\]
where $M^\PP$ is
an $({\cal F},\PP)$-martingale on $[0,1]$ and $(A^\PP,B^\PP)$ is ${\cal F}$-adapted and $\PP$-$\as$ absolutely continuous with respect to time. In particular, $\PP$ is said to be have characteristics $(\mu,\sigma^2)(\PP)$, which are defined in the following way,
\[
    \quad \mu =\frac{\diff A^\PP_t}{\diff t}, \sigma^2_t=\frac{\diff B^\PP_t}{\diff t},
\]

Note that $(\mu,\sigma^2)$ is ${\cal F}$-adapted and determined up to $d\PP\times dt$, almost everywhere.
We now let $F(t,x,a,b)$ be convex with respect to $(a,b)$ for every $(t,x)$, and seek for
\begin{equation*}
    \inf_{\PP \in \Pc} \E_\PP \int_0^T F(t, X_t, \mu_t, \sigma_t^2) \, \diff t,
\end{equation*}
An  example of modelling function $F$ is  to enforce a Martingale constraint  on $X$ trough 
       \begin{equation*}
              F( \mu) = \begin{cases}
                  0       & \text{if } \mu = 0 \\
                  +\infty & \text{otherwise}.
              \end{cases}
          \end{equation*}
          or its soft version  ($ \lambda >>1$): 
                 \begin{equation*}
              F( \mu) =  \lambda   \| \mu \|^2
                        \end{equation*}

At this stage, the processes $\mu, \sigma$ are very general and can be generally path-dependent, however, as showed in \cite{guoCalibrationLocalStochasticVolatility2021a}, they can be chosen as local processes, i.e. functions of $(t,X_t)$ only:
indeed, for any choice of $\mu,\sigma$, there exists a local version $\mu(t,x),\sigma(t,x)$ that preserves the constraints (i.e. option prices) and that can only reduce the cost \eqref{eq:cont_mot}.
The minimization problem can therefore be formally reduced to
$X_t$ solutions of the stochastic differential equation:
\begin{equation}
    \label{eq:cmot_sde}
    \diff X_t = \mu_t(X_t) \, \diff t + \sigma_t(X_t) \, {\diff B}_t,
\end{equation}
where $B_t$ is a standard Brownian motion.  \\

The calibration problem adds  a set of  discrete constraints indexed by $i \in \Ic := \{1, \dots, N_c\}$ described by $(\tau_i, c_i, G_i)_{i \in \Ic}$ where for each $i$ the  triplet $(\tau_i, c_i, G_i)$ is the maturity, the price and the payoff function of an observed derivative price on the market.
\\
We will seek for an element $\PP \in \Pc$ such that
\begin{equation}
    \label{price}
    \expect_{\PP}{G_i(X_{\tau_i})} = c_i.
\end{equation} 
As an example, calibrating a set of call options at a fixed  maturity $t$ would lead to $G_i(x) = (x - K_i)^+$, where $K_i$ is the strike  of the $i$-th option, and $\tau_i = t$ for all $i$.
Moreover, we assume that there is a  set of listed maturities $\{ \tau_i \}_{i \in \Ic} $.  Because we will use $\{ t_k\} _{k=0}^{N_T} $ a time discretisation of $[0,T]$ we assume for simplifications  that the maturities take values in this discrete set.  Thus the set of constraints  $\Ic$ can be partitioned as $\Ic = \bigcup_k \Ic_k, \, \Ic_k \coloneqq \{i \in \Ic, \tau_i = t_k\}$, where $t_k$ is the $k$-th distinct maturity. \Nte{ By convention we set $\Ic_k  = \emptyset $ if there are no constraints at time $t_k$. } 
To take these constraints in account in the the minimisation problem, we will use  for all $k$, the costs: 
 \begin{equation*}
             C_k( \{ g_i \}) = \begin{cases}
                  0       & \text{if }  g_i = c_i, \forall i \in \Ic_k\\
                  +\infty & \text{otherwise}.
              \end{cases}
          \end{equation*}
          or  its mollified version ($\lambda  >>1$): 
           \begin{equation}
           \label{CK}  
              C_k(  \{ g_i \})  =   \lambda \sum_{i \in \Ic_k}   \| g_i - c_i  \|^2 .
                        \end{equation}           

Using these notations, the calibration problem becomes:
\begin{equation}
    \label{eq:cont_mot}
    \inf_{\PP \in \Pc} \E_\PP  \left[ \int_0^T F(t, X_t, \mu_t, \sigma_t^2) \, \diff t   \right]   + \sum_{k=0}^{N_K-1}  C_k( \{ \expect { G_i(X_{\tau_i}) } \}_{i \in \Ic_k}  )  + M_0({X_0}_\# \PP)
\end{equation}
Where $M_0 \coloneqq \delta_{\nu_0} $  enforces  the initial distribution of prices  $\nu_0$. \\ 


The cost function  can also be used to constrain  the model volatility and/or its regularity. 
We might want for instance $\sigma$  to be close to a prescribed guess $\bs$, which can be enforced  by adding  to $F$ \cite{loeper2016option} 
    \[
    G= (\sigma-\bar\sigma)^2
    \](which is convex in $\sigma^2$ and is linked to the Bass martingale problem \cite{backhoffM}), or 
          in  \cite{guo2021optimal} 
          \begin{equation}
              \label{GR}
              G(\sigma^2) = a \left(\frac{\sigma^2}{\bs^2}\right)^p + b \left(\frac{\sigma^2}{\bs^2}\right)^{-q} + c
          \end{equation}
          with $p, q, a, b > 0$ and $c \in \R$ such that $F$ is convex with minimum  at $\sigma = \bs$. It is also a barrier as $\sigma^2$ goes to $0$.
          
          In this paper, as explained in section \ref{sec:specentr},   we use 
          \begin{equation*}
            G( \sigma^2) =  \frac{\sigma^2}{\bs^2} - 1 - \log{\frac{\sigma^2}{\bs^2}} 
          \end{equation*}
          corresponding to the integrand of the Specific entropy  \eqref{specrelent}.
         It will allow to discretize (in time) the problem as
          a multi-marginal EOT problem.
          It is again convex with minimum  at $\sigma^2 = \bs^2 $ and a
          barrrier as $\sigma^2$ goes to $0$ but unlike (\ref{GR}) it is only sublinear
          as $\sigma \nearrow + \infty$ (see  \cite{Mpaper} where we need additional  assumptions to prove a convergence result based on the 
          the discretisation of $F_r$ as (\ref{SRE}) below).

\section{Time Discretisation as  a Multi-Marginal Entropic Martingale Transport}

\subsection{Notations}

We will discretize our problem in time, replacing the interval $[0, T]$ with a regular grid of $N_T + 1$ timesteps,  $t_k = k\,h$ for $k \in \{0, \dots, N_T\} =: \Kc^h$, where $h := T/N_T$ is the time step.
We impose that all the calibration times $\tau_i$ are included in the grid, i.e. $\tau_i = t_{k_i}$ for some $k_i \in \Kc^h$.

Instead of functions $t \mapsto \omega(t)$, we consider their discrete counterparts, which are n-tuples $(\omega_0, \dots, \omega_{N_T}) \in \R^{N_T+1}$, in which $\omega_k$ corresponds to the value of the path at time $t_k$. Instead of $\R^{N_T}$, we denote by $\Xc_k$ the space of values that $\omega_k$ can take, and by $\Omega^h := \Pi_{i=0}^{N_T} \Xc_i$ the space of discrete paths\footnote{ \Nte{
An element $(t \mapsto \omega(t)) \in \Omega$ is hence replaced by a n-tuple $(\omega_0, \dots, \omega_{N_T}) \in \Omega^h$ with $\omega_k \in \Xc_k$ for $k \in \Kc^h$.}}.
We denote  $\Pc^h$ the set of probability measures on $\Omega^h$.

We are hence searching for a probability measure $\PP^h$ on $\Omega^h$.
We denote by $(X_k)_{k \in \Kc^h}$ the canonical process of $\PP^h$ on $\Omega^h$. We will denote by $\PP^h_k := {X_k}\#\PP^h$ the marginal law of $\PP^h$ at timestep $k \in \Kc^h$, and by $\PP^h_{k,l} := (X_k, X_l)\#\PP^h$ the joint law between time steps $k \in \Kc^h$ and $l \in \Kc^h$.

We note $\Kc^h_{-i} = \Kc^h\setminus \{i\}$ the set of timesteps except timestep $i$, and $\diff x_{-i} = \prod_{\Kc^h_{-i}} \diff x_k$, which allows to write the marginal law $\PP^h_{k}$ as $\PP^h_{k} = \dint \PP^h(x_k, \diff x_{-k})$ and joint laws in a similar fashion.
Similarly, we note $\diff x_{[i, j]} = \Pi_{k=i}^j \diff x_k$.

We note $\nu_0$ the initial marginal of our process, which is imposed, $X_0 \sim \nu_0$. It may or may not be a Dirac mass in our case.

We denote by $\bPP^h$ the reference measure on $\Omega^h$ that we will use to regularize the problem. We will denote by $(Y_k)_{k \in \Kc^h}$ the canonical process of $\bPP^h$ on $\Omega^h$. It's law is determined by a Euler-Maruyama discretisation of the continuous reference process :
\[
    Y_{k+1} = Y_k + \overline{\mu}(Y_k, kh) \, h + \overline{\sigma}(Y_k, kh) \, h^{1/2} \, Z_k, \quad \forall k \in \Kc^h_{-0}, \, Y_0 \sim \rho_0, \quad Z_k \sim {\mathcal N}(0,1) 
\]

For any probability measure $\PP^h$ and $\bPP^h$  in $\Pc^h$, we note $\KL(\PP^h|\bPP) = \expect_{\PP^h}*{\log\left(\frac{\diff \PP^h}{\diff \bPP^h}\right) - 1}$ the
Kullback-Leibler divergence between $\PP^h$ and $\bPP^h$ if $\PP^h \ll \bPP^h$. By convention, if $\PP^h \not\ll \bPP^h$, we set $\KL(\PP^h|\bPP^h) = +\infty$.


\subsection{Time discretisation of conditional moments}
\label{sec:prelim}
As opposed to the continuous-time approach, which uses Markovian projections of the processes, and as such the variable being optimised are functions representing the drift and volatility, in this discrete-time approach, we will directly optimise on $\PP^h$. 
In order to  justify the choice of moment variables in the discrete problem, let us consider the Euler-Maruyama discretization of a diffusion process. Let $X_t$ be a diffusion process  with drift $\mu$ and volatility $\sigma$, following the SDE;
\[
    \diff X_t = \mu(X_t, t) \, \diff t + \sigma(X_t, t) \, {\diff \RR}_t;
\]
and consider the Euler-Maruyama time  discretization of the process:
\[
    X^h_{k+1} = X^h_{k} + \mu(X^h_k, kh) \, h + \sigma(X^h_k, kh) \, h^{1/2} \, Z_k
\]
where $\forall k \in \{0, \dots, N_T\} := \Kc^h, \, Z_k$ is a standard normal random variable, and  we note $\PP^h$ the law of $(X^h_k)_{k \in \Kc^h}$. We have $\PP^h \in \Pc^h_{\text{EM}}$.

For such a process, we can compute the following quantities from conditional expectations :
\begin{align}
    \label{beta_def} \beta_k(x)   & = \frac{1}{h} \expect*{X_{k+1} - X_k | X_k = x} = \mu(x, kh),                                                                      \\
    \label{alpha_def} \alpha_k(x) & = \frac{1}{h} \expect*{(X_{k+1} - X_k)^2 | X_k = x} = \mu^2(x, kh) \, h + \sigma^2(x, kh) \xrightarrow[h \to 0]{} \sigma^2(x, kh).
\end{align}
These variables are computed from the law $\PP^h$ and are the discrete counterpart of the drift and volatility of the continuous process. They can hence be used to  discretize of $F(t, X_t, \mu, \sigma^2)$. \\

The method actually extends to any  vector of variables $ (b_k) : \Xc \to \R^K$ defined by taking the conditional expectation of a  function $B : (\Xc, \Xc) \to \R^K$ to be specified and depending on two consecutive timesteps :
\begin{equation} 
\label{general} 
    b_k(x) = \frac{1}{h} \expect{B(X_{k}, X_{k+1}) | X_k=x}.
\end{equation} 
In the present  section  the  variables $\beta_k$ and $\alpha_k$ correspond to  the function $B(X, Y) = \begin{bmatrix}
        (Y-X) \\ 1/2 \, (Y-X)^2
    \end{bmatrix}$, 
but  in section \ref{alsv}, we will use $B(X, Y) = (1-e^{Y-X})$.

\subsection{Specific relative entropy}
\label{sec:specentr}
We give a formal derivation of the Specific Entropy (see
\cite{Gantert91}   \cite{follmer22} \cite{backhoff23} for rigorous smoothnesss hypothesis and proofs).\\

The Kullback-Leibler divergence between two normal laws $\mathcal{N}(\mu_1, \sigma_1^2)$ and $\mathcal{N}(\mu_2, \sigma_2^2)$ is equal to :
\begin{equation*}
    \KL(\mathcal{N}(\mu_1, \sigma_1^2) | \mathcal{N}(\mu_2, \sigma_2^2)) = \frac{1}{2}\left(\frac{\sigma_1^2 + (\mu_1 - \mu_2)^2}{\sigma_2^2} - 1 - \log\left(\frac{\sigma_1^2}{\sigma_2^2}\right)\right).
\end{equation*}

Consider two diffusion measures $\PP$ and $\bPP$ defined by the following SDEs on their respective canonical processes $X$ and $Y$ :
\begin{align}
    \diff X_t & = \mu(X_t, t) \, \diff t + \sigma(X_t, t)  \, {\diff \RR}_t, \,  X_0 \sim \PP_0,                         \\[8pt]
    \diff Y_t & = \overline{\mu} (Y_t, t) \,  \diff t + \overline{\sigma}(Y_t, t) \, {\diff \RR}_t,  \, Y_0 \sim \PP_0 .
\end{align}
We can discretize on a grid of step $h$ as law $\PP^h$ and $\bPP^h$ using the Euler-Maruyama discretization from previous section, giving their respective canonical processes $X^h$ and $Y^h$ :
\begin{align}
    X^h_{k+1} & = X^h_{k} + \mu(X^h_k, kh) \,  h + \sigma(X^h_k, kh) \,  h^{1/2}  \, Z_k, \quad \forall k \in \Kc^h, \,
    X^h_0 \sim \rho_0     ,                                                                                                                                \\[8pt]
    Y^h_{k+1} & = Y^h_{k} + \overline{\mu}(Y^h_k, kh) \, h + \overline{\sigma}(Y^h_k, kh) \, h^{1/2} \, Z_k, \quad \forall k \in \Kc^h, Y^h_0 \sim \PP_0.
\end{align}
Noting $\mu_k(x) = \mu(x, kh)$, $\sigma_k(x) = \sigma(x, kh)$, and $\overline{\mu}_k(x) = \overline{\mu}(x, kh)$ and $\overline{\sigma}_k(x) = \overline{\sigma}(x, kh)$, 
 the Kullback-Leibler divergence can then be decomposed as follows :
\begin{equation}
    \label{SRE}
    \begin{array}{ll}

        h \, \KL(\PP^h|\bPP^h) & = h\, \sum_{k=0}^{N_T - 1} \dint \KL(\mathcal{N}(\mu_k(x) h, \sigma_k(x)^2 h) | \mathcal{N}(\overline{\mu}_k(x) h, \overline{\sigma}_k(x)^2 h) ) \,  \PP^h_k(dx)                                                                                                                          \\[12pt]
                               & = \frac{h}{2}  \, \sum_{k=0}^{N_T - 1} \dint  \left(\dfrac{\sigma_k(x)^2 h + ((\mu_k(x)-\overline{\mu}_k(x))h)^2}{\overline{\sigma}_k(x)^2h} - 1 - \log\left(\dfrac{\sigma_k(x)^2 h}{\overline{\sigma}_k(x)^2h}\right)\right) \, \PP^h_k(dx) \\[12pt]
                               & = \frac{h}{2} \,  \sum_{k=0}^{N_T - 1} \dint\left(\dfrac{\sigma_k(x)^2 + (\mu_k(x) - \overline{\mu}_k(x))^2h}{\overline{\sigma}_k(x)^2} - 1 - \log\left(\dfrac{\sigma_k(x)^2}{\overline{\sigma}_k(x)^2}\right)\right) \, \PP^h_k(dx)       \\[12pt]
                               & \xrightarrow[h \to 0]{} \mathcal{S}(\PP | \bPP ) := \frac{1}{2} \dint \expect_{\PP}*{\dfrac{\sigma_t^2}{\overline{\sigma}^2} - 1 - \log\left(\dfrac{\sigma_t^2}{\overline{\sigma}^2}\right)} \diff t   .
    \end{array}
\end{equation}

This convergence result motivates the use the Specific entropy as a regulariser of the continuous problem, because it is linked to a natural discretization in terms of the Kullback-Leibler divergence, and therefore 
and entropy regularization of a discretisation if (\ref{eq:SMOT}). 

\subsection{Discretization and specific entropy regularisation}
\label{sec:discretisation}

The cost function, simplified for the presentation is  
\begin{equation*}
 \Fc (\PP) =    \expect_{\PP}*{  \dint_0^T F(\mu_t,\sigma_t) \diff t  } +  \sum_{k =0}^{N_T-1}      C_k( \{ \expect { G_i(X_{\tau_i}) } \}_{i \in \Ic_k}  )  + M_0(\PP_0) .
\end{equation*}

The  specific entropic regularisation  of the continuous problem (\ref{eq:cont_mot}) is 
\begin{equation*}
  \inf_{\PP \in \Pc}   \Fc (\PP)   + \mathcal{S}(\PP | \bPP )  
\end{equation*}
where $\mathcal{S}$ the specific entropy introduced in section \ref{sec:specentr}.  Its 
   time discretisation (as in (\ref{P}) and (\ref{Ph}))  
using (\ref{SRE}) and the variables $\beta_k$ and $\alpha_k$ defined in equations (\ref{beta_def}) and (\ref{alpha_def})
  can  be  written as: 
\begin{equation}
    \label{mmmtprim}
    \inf_{\substack{\PP^h \in \Pc^h}}  \Fc^h(\PP^h)   + h \KL(\PP^h|\bPP^h)
\end{equation}
with 
\[
      \Fc^h(\PP^h) =  h  \,  \expect_{\PP^h}*{  \sum_{k=0}^{N_T-1} F( \beta_k (X_k),\alpha_k(X_k)) }   + \sum_{k =0}^{N_K-1}   C_k( \{ \expect { G_i(X_{\tau_i}) } \}_{i \in \Ic_k})   + M_0({X_0}_\#\PP^h) . 
\] 
or 
\begin{equation}
\label{FH} 
      \Fc^h(\PP^h) =  h  \,  \expect_{\PP^h}*{     \sum_{k=0}^{N_T-1} F( b_k (X_k))  }  +   \sum_{k =0}^{N_K-1}   C_k( \{ \expect { G_i(X_{\tau_i}) } \}_{i \in \Ic_k})    + M_0({X_0}_\#\PP^h) ,
\end{equation}
using, for simplification, the general form  of conditional moments  $ b_k  = \frac{1}{h}\expect_{\PP^h}*{B(X_k^h, X_{k+1}^h)|X_k^h}$ (section \ref{sec:prelim}). \\

As customary in Entropic OT, we  perform a change of variable in the definition (\ref{FH})  in order to apply Fenchel-Rockafellar convex duality (section \ref{ds}) and 
later use Sinkhorn algorithm (section \ref{ss}).  
Let us define, for all $k$: 
\begin{definition}
    \label{rem:characteristics}
    \begin{align*}
        \nu_k  (dx)  & = X_k\#\PP^h                (dx)                               \\
     m_k(dx)  &  =   \nu_k(dx)  b_k(dx) = \frac{1}{h} \expect_{\PP^h_{k,k+1}} {B(dx, X_{k+1}^h)}       \\
                 g_k            & = \{  \expect_{\PP^h_{\tau_i}}*{G_i({X_{\tau_i} }_\# \PP^h)} \}_{  i \in \Ic_k} 
    \end{align*}
\end{definition} 
Notice that  the operator 
\begin{definition}
\label{DeltaD} 
 \begin{equation*}
        \begin{array}{lrcl}
            \Delta^\dagger : & \Pc(\Omega^h) & \to & (\otimes_{k=1}^{N_T-1} \Pc(\Xc_k))  \times  (\otimes_{k=1}^{N_T-1} \Mc(\Xc_k)) \times
           (\otimes_{k=1}^{N_T-1}   \R^{\# \Ic_k}     )                                                                                                         \\[8pt]
                             & \PP^h & \rightarrow & \Delta^\dagger\PP^h : = \left( \{ h\, \nu_k \} , \{ h\,  m_k\} ,\{  g_k\}   \right)
        \end{array}
    \end{equation*}
    is linear. \Nte{ Remember that  $ \#Ic_k$  the cardinal of the set of maturities at time $t_k$ and 
$ \Ic_k = \emptyset$ if there are none} 
    ($ \Mc(\Xc_k))$ is the set of  vector Radon measure on $ \Xc_k$).   The $h$ multiplicative factor is used to simplify the dual formulation. 
    
\end{definition}    

Overloading the notation  $\Fc^h$ as a function of the new variables $\left( \{ \nu_k \} , \{ m_k\} ,\{  g_k\}   \right) $ we obtain 
\begin{equation}
\label{FH2} 
      \Fc^h( \Delta^\dagger \PP^h) =    \sum_{k=0}^{N_T-1}   \left(  h \,  \expect_{\nu_k}*{  F( \frac{  h\, m_k (X_k)}{ h\,  \nu_k(X_k)})  }    +          C_k(g_k)  \right)  + M_0(\nu_0) .
\end{equation}
and $\Fc^h$ is a convex function of $\left( \{ \nu_k \} , \{ m_k\} ,\{  g_k\}   \right)$  by construction (we   recognise in the integrand the perpective functions of $F$ which is assumed to be convex). \\

In summary, (\ref{mmmtprim}) can be casted  in the Fenchel Rockaffelar convex setting (appendix \ref{FRC})  with a primal problem:
\begin{equation}
   \inf_{\substack{\PP^h \in \Pc^h}}  \Fc^h( \Delta^\dagger \PP^h)   + \Gc(\PP^h) 
     \label{frprim}, 
\end{equation} 
where the relative entropy 
   \begin{equation}
    \label{G} 
        \Gc(\PP^h) :=  h \KL(\PP^h|\bPP^h).
    \end{equation}
is also convex.     

\begin{remark}[Convergence as $h \rightarrow 0$] 
We point out  that while the solution of this problem is a measure that respects the initial condition and the price constraints similarly to the continuous problem,
the constraint that $\PP^h$  is the law of a  discrete diffusion Markov chain and  belongs to the set of semi-martingale laws ${\cal P}^0_s$   in the limit  has been relaxed.  We show  in \cite{Mpaper} under additional conditions  that this relaxation is tight under additional regularity assumptions and 
 that a minimizer of the continuous problem (\ref{eq:cont_mot})  can be constructed using a sequence in $h$ of minimizers of (\ref{mmmtprim}).
 \end{remark} 
 

\subsection{Dual Problem}
\label{ds}

We first need to compute the Legendre transforms of $\Fc$ and $\Gc$, and the adjoint operator of $\Delta^\dagger$.

\begin{lemma}
    $\Gc^\star$, the Legendre-Fenchel transform of $\Gc$ is given by :
    \[
        \begin{array}{lrcl}
            \Gc^\star : &  \C_b(\Omega^h)  & \to & \overline{\R}                                                \\[8pt]
                        & f & \rightarrow & h \, \E_{\bPP^h}(\exp(f/h))
        \end{array}
    \]
\end{lemma}
\begin{lemma}
\label{LLL} 
 $\Fc^{h,\star}$ 
    the Legendre transform of $\Fc^h$ (\ref{FH2})   is given by :
        \[
        \begin{array}{ll}
       
            \Fc^{h,\star} : &   (\otimes_{k=1}^{N_T-1} \C_b(\Xc_k))  \times  (\otimes_{k=1}^{N_T-1} \C_b(\Xc_k)) \times
             (\otimes_{k=1}^{N_T-1}   \R^{\# \Ic_k}     ) 
                    \to  \overline{\R}                                                \\[8pt]
                        &   \Phi  \coloneqq   ( \{ \phi_{\nu_k} \} , \{ \phi_{m_k} \}  , \{\Lambda_k \} )   \rightarrow  \sum_{k=0}^{N_T-1}   M_k^\star(\phi_{\nu_k}  +  F^\star (\phi_{m_k}) )  + \C_k^\star(\Lambda_k) 
        \end{array}
    \]
 where $M^\star_0$ is the Legendre-Fenchel transform of $M^0$ the function enforcing the initial law and $M_k^\star$ for $k= 1,N_T-1$ is the indicatrix $\iota_{\phi \le 0}$ or
equivalently the  Legendre-Fenchel transform of $ M_k (\nu) \coloneqq \iota_{\nu \ge 0}$.
\end{lemma}
\begin{proof}
The proof is a direct application of the definition of the Legendre Fenchel transform. Note that 
 $\Fc^h$ is separable in $k$  and position $x_k$. 
We first rewrite the cost function in a more general form   
  \[
      \Fc^h( \Delta^\dagger \PP^h) =    \sum_{k=0}^{N_T-1}   \left(  h \,  \expect_{\nu_k}*{  F( \frac{ h\, m_k (X_k)}{ \ h\, \nu_k(X_k)})}  + M_k(\nu_k)      +          C_k(g_k)  \right)  
\]
(\ref{FH2}) corresponds to  $M_k(\nu_k) =  \iota_{\nu_k \ge 0}$ for $k= 1,N_T-1$ and does not change the cost as we identify $(m, \nu) \rightarrow   \nu F( m / \nu) $ with its bi-dual 
$ \{ (m, \nu) \rightarrow \nu F ({m}/ {\nu}) \}^{\star \star}     =  \nu F( \frac{  m }{ \nu}) $ if $ \nu >0$, $0$ if $ (m,\nu) = (0,0)$ and $+\infty $ else.  We also use the classic characterisation 
of the LF dual of a perspective function 
$ \{ (m, \nu) \rightarrow \nu F ({m}/ {\nu}) \}^{\star} ( \phi,\psi) = \iota( \phi + F^\star( \psi)) $ .

\end{proof}

\begin{lemma}
\label{dpa} 
    The adjoint operator to $\Delta^\dagger$ is given as 
        \[
        \begin{array}{ll}
       
            \Delta  : &   (\otimes_{k=1}^{N_T-1} \C_b(\Xc_k))  \times  (\otimes_{k=1}^{N_T-1} \C_b(\Xc_k)) \times
             (\otimes_{k=1}^{N_T-1}   \R^{\# \Ic_k}     ) 
                    \to                   \C_b(\Omega^h)                             \\[8pt]
                        & \Phi  \coloneqq  \{ \phi_{\nu_k} \} , \{ \phi_{m_k} \} , \{\Lambda_k \}   \rightarrow   
                         \sum_{k=0}^{N_T-1}  \left( 
                         h \phi_{\nu_k} +    \phi_{m_k}   \cdot B(x_{k}, x_{k+1}) +   \Lambda_k    \cdot   \{ G_i(x_{\tau_{i}}) \}_{i \in \Ic_k}    \right) 
                                \end{array}
    \]

\end{lemma}

We apply  the Fenchel-Rockafellar duality (appendix \ref{FRC}) to get the dual of (\ref{frprim}) :

\begin{equation} 
\label{mydual} 
\sup_{\Phi}   - \Fc^{h,\star} ( - \Phi)  -  h \,  \E_{\bPP^h}(\exp( \dfrac{ \Delta \Phi }{h} ))
   \end{equation}
 
\begin{remark} 
\label{R1}  
A simplification is possible in the setting of lemma \ref{LLL}. We have   
have  $M_k^\star =  \iota_{\phi \le 0 } $ for $k= 1,N_T-1$ enforcing 
  $-\phi_{\nu_k} +  F^\star(-\phi_{m_k}) \le  0$.  Because  $-  h \,  \E_{\bPP^h}(\exp(  {\Delta \Phi }/{h}) $ is decreasing in $\phi_{\nu_k} $, 
  this constraint necessarily saturates as $ \Delta \Phi $ is increasing in $\phi_{\nu_k}$. We 
 can set $M_k^\star( \phi_{\nu_k} ) = 0 $, and replace  the variable $\phi_{\nu_k} $ by $  F^\star(-\phi_{m_k})$ in $ \Delta \Phi $. 
 
 Conversely, $M_k^\star $ can be kept as the Legendre Fenchel transform of a general convex function $M_k$ acting on $\nu_k$  and to 
 be added to $F$. \Nte{See the proof of lemma \ref{LLL}.} 
\end{remark} 


\begin{proposition}
    Assuming  the supremum in (\ref{mydual}) is attained by  $\Phi$,   then $\Phi$ induces a measure $\PP^{h}$ through
    \label{prop:optimaldensity}
    \begin{equation*}
        \frac{\diff \PP^{h}}{\diff \bPP^h } = \exp\left(\frac{\Delta \Phi}{h}\right),
    \end{equation*}
    $\PP^{h}$ is the optimal solution of (\ref{frprim})
    for the constraints functions $M,C$ that are finite under $\PP^{h}$.
\end{proposition}

\begin{proof}
We apply theorem \ref{tfr}. 
    The first optimality condition is given by:
    \begin{equation*}
        \Delta \Phi  \in h \, \partial_{\PP^{h }} \KL(\PP^{h}|\bPP^h)
    \end{equation*}
    which leads to :
    \begin{equation}
        \Delta \Phi  = h \, \log\left(\frac{\diff \PP^{h}}{
            \diff \bPP^h}\right)
    \end{equation}
\end{proof}

We can therefore  identifiy the  law of  $\PP^h$  and its marginals as functions of the dual potentials.
We  use the folowing factorisation which   will also be useful to implement Sinkhorn algorithm below. 

\begin{definition}
    For $i=0,N_T-1$, let $\Delta_{i, i+1}$ the transitional part of $\Delta$ ( see definition \ref{dpa}) depending only on $  \phi_{m_i}$:
    \begin{equation*}
        \Delta_{i, i+1} (\phi_{m_i})  [ x_i, x_{i+1} ]  = B(x_{i}, x_{i+1}) \cdot  \phi_{m_i}(x_i),
    \end{equation*}
    and use the following simplified notation $ G_k \coloneqq \{ G_i(x_{\tau_{i}}) \}_{i \in \Ic_k} $.
    Let $\psi^u_i, \psi^d_i$ potential functions factorizing the kernel after  (Down) and before (Up) time  $t_i$:
    \begin{align*}
        \psi^u_i(x_i) & = \log\left(\int \rho_0 \prod_{k=0}^{i-1} \exp\left(\frac{\Delta_{k,k+1} ( \phi_{m_k}) }{h} + \phi_{\nu_k} + 
        \frac{ \vv{\Lambda_k} \cdot \vv{G_k} }{h}  \right) \bPP_{k,k+1}^h \diff x_{[0, i-1]}\right)         \\
        \psi^d_i(x_i) & = \log\left(\int \prod_{k=i}^{N_T-1} \exp\left(\frac{\Delta_{k,k+1} ( \phi_{m_k}) }{h} + \phi_{\nu_{k+1}} + 
         \frac{ \vv{\Lambda_{k+1}} \cdot \vv{G_{k+1}} }{h}\right) \bPP_{k,k+1}^h \diff x_{[i+1, N_T]}\right)
    \end{align*}
    \end{definition}
\begin{proposition}
    \label{prop_markovian}
   Let $\PP^{h}$ an optimal solution of \eqref{frprim}. The following properties hold true :
    \begin{itemize}
        \item Its density  of the joint law at time $t_i$ and $t_{i+1}$  with respect to the reference measure is given by :
              \begin{align*}
                \PP^{h}_{i, i+1}(x_i, x_{i+1})   = & \exp  [  \psi^u_i(x_i) + \phi_{\nu_{i}}(x_i) +  ( \vv{\Lambda_i} \cdot \vv{G_i}    )/h                                                     \\
                                                               & \hphantom{\exp(} + \Delta_{i, i+1} ( \phi_{m_i})  [x_i, x_{i+1} ] /h                                                                                       \\
                                                               & \hphantom{\exp(} +  ( \vv{\Lambda_{i+1}} \cdot \vv{G_{i+1}} )/h  + \phi_{\nu_{i+1}}(x_{i+1}) + \psi^d_{i+1}(x_{i+1})
                                                            ]     \,\,   \bPP_{i, i+1}^h( x_i, x_{i+1})                                 
                                                                   \end{align*}
                     \item Its marginal are given by :
              \begin{align*}
                  \nu_k (x_k) = \PP^{h}_k(x_k) & = \exp(\psi^u_k(x_k) + \phi_{\nu_k}(x_k) +  ( \vv{\Lambda_k} \cdot \vv{G_k} )/h  + \psi^d_k(x_k))
              \end{align*}
    \end{itemize}

\end{proposition}
\begin{proof}
    In Appendix \ref{proof_propmarkov}.
    \Nte{Note we are  abusing notations using the same variables for the measures and their densities. }
    
\end{proof}

The following proposition is important for the implementation of Sinkhorn algorithm.
\begin{proposition}
\label{ud} 
Remark that the  quantities $\psi^u_k$, $\psi^d_k$ can be computed recursively :
    \begin{align*}
        \psi^{u}_{k+1} & = \log\left(\int \exp\left(\psi^{u}_k + \frac{\Delta_{k,k+1}  ( \phi_{m_k})}{h}  + \phi_{\nu_k} + ( \vv{\Lambda_k} \cdot \vv{G_k} )/h  \right) \bPP_{k,k+1}^h \diff x_k\right)       \\
        \psi^{d}_{k-1} & = \log\left(\int \exp\left(\psi^{d}_k + \frac{\Delta_{k-1,k}  ( \phi_{m_{k-1} })}{h}  + \phi_{\nu_{k}} + ( \vv{\Lambda_{k}} \cdot \vv{G_{k}} )/h  \right) \bPP_{k-1,k}^h \diff x_{k}\right)
    \end{align*}
    with $\psi^u_0 = \log \nu_0$ and $\psi^d_{N_T} = 0$.
\end{proposition}
\begin{proof}
    Let $i \in \Kc$ and $x_i \in \Xc_i$. We can compute the following integral :
    \begin{align*}
        \psi^u_i(x_i) & = \log\left(\int \rho_0 \prod_{k=0}^{i-1} \exp\left(\frac{\Delta_{k,k+1}}{h}  ( \phi_{m_k}) + \phi_{\nu_k} +  ( \Lambda_k \cdot \vv{G_k})/h  \right) \bPP_{k,k+1}^h \diff x_{[0, i-1]}\right)                    \\
                      & = \log\left(\int \left(\int \rho_0 \prod_{k=0}^{i-2} \exp\left(\frac{\Delta_{k,k+1}}{h}  ( \phi_{m_k})  + \phi_{\nu_k} + (\Lambda_k \cdot \vv{G_k})/h  \right) \bPP_{k,k+1}^h \diff x_{[0, i-2]}\right) \right. \\
                      & \hphantom{= \log\left(\int \right.} \left. \exp\left(\frac{\Delta_{i-1,i}  ( \phi_{m_i})}{h} + \phi_{\nu_{i-1}} + ( \Lambda_{i-1} \cdot \vv{G_{i-1}})/h  \right) \bPP_{i-1,i}^h \diff x_{i-1}\right)              \\
                      & = \log\left(\int \exp\left(\psi^{u}_{i-1} + \frac{\Delta_{i-1,i}  ( \phi_{m_i})}{h} + \phi_{\nu_{i-1}} + ( \Lambda_{i-1} \cdot \vv{G_{i-1}})/h \right)\bPP_{i-1,i}^h \diff x_{i-1}\right).
    \end{align*}
    A symmetric computation  can be done for $\psi^d_i$.
\end{proof}

\section{Numerical method}
\label{ss}

\subsection{Sinkhorn Algorithm } 

Sinkhorn algorithm is based on the iterative ``\`a la Gauss-Seidel'' resolution of the optimality conditions of the dual problem (\ref{mydual}).
This amount  to coordinate-wise (in $\Phi$) maximisation. We use   the notations introduced in proposition \ref{ud}, the additional index $n$ 
corresponds to the   Sinkhorn iterations rank: 
\begin{equation}
    \label{sinkhorn}
    \left\{  \begin{array}{ll}     
        \phi_{\nu_k}^{n+1} = \arg\,\sup_{\phi}   & -M_k^\star(-\phi +  F^\star( -\phi^n_{m_k} ))      -   \ldots  \\[10pt]                                                                                                                                                                                                                         
      &    h\expect_{\bPP_k^h}*{\exp(\psi^{u,n+1}_k + \phi +  ( \vv{\Lambda^n_k} \cdot \vv{G_k} )/h + \psi^{d,n}_k)}                \\[10pt]                      
        \vv{\Lambda^{n+1}_k} = \arg\,\sup_{\vv{\Lambda}} &  -C_k^\star(-\vv{\Lambda}) - \ldots  \\[10pt] 
        &  h\expect_{\bPP_k^h}*{\exp(\psi^{u,n+1}_k + \phi^{n+1}_{\nu_k} +  ( \vv{\Lambda} \cdot \vv{G_k} )/h  + \psi^{d,n}_k)}                \\[10pt]
        \phi_{m_k}^{n+1} = \arg\,\sup_{\phi}   &  -M_k^\star(-\phi_{\nu_k} +  F^\star( -\phi))                                                                                          \\[10pt]
    &   - h\expect_{\bPP_{k,k+1}^h(x,.) }*{\exp(\psi^{u,n+1}_k + \phi_{\nu_k}^{n+1} + ( \vv{\Lambda^{n+1}_k} \cdot \vv{G_k})/h + \Delta_{k,k+1} (  \phi )   + \psi^{d,n}_k)} 
    \end{array} \right.
\end{equation}
and this is  $\forall  k \in [0,N_T-1]$. 
 Each subproblem is a relaxation of the global strictly concave problem and therefore is well posed.   Unlike in the classic OT marginal constraints 
 case there are no explicit formula for the optimal potentials and it is necessary to use a Newton method to solve the  optimality conditions equations. 
The maximisations can  be performed point-wise in $x_k$ for the potentials $(phi_{\nu_k},  \phi_{m_k})$  and on the vectors $(\Lambda_k)$. \\

The algorithm is decomposed as follows: 

\begin{algorithm}[H]
    \caption{Sinkhorn algorithm for problem (\ref{mydual})}
    \label{algo:sinkhorn}
    \KwIn{Number of timesteps $N_T$, support of each space $\Xc_k^{\diff x}$}
    \KwIn{Stopping tolerance $\epsilon$, reference measure $\bPP^h$}
    \KwIn{Initial potentials $\phi_{\nu_k}^0$, $\phi_{b_k}^0$, $\Lambda_i^0$}
    \KwResult{Numerical solution of problem (\ref{mydual})}
    $\psi^{u,0}_0 \gets \log \nu_0$\;
    $\psi^{d,0}_{N_T} \gets 0$\;
    \For{$n \gets 0$ \KwTo $N$}{
        \For{$k \gets N_T-1$ \KwTo $0$}{
            $\psi^{d,n}_k \gets$ UpdatePsiDown($\psi^{d,n}_k$, $\phi_{\nu_k}^{n}$, $\phi_{b_k}^{n}$, $\Lambda^n_i$, $\bPP^h$)\;
        }
        \For{$k \gets 0$ \KwTo $N_T - 1$}{
            $\phi_{\nu_k}^{n+1} \gets$ SolveMarginal($\phi_{\nu_k}^{n}$, $\phi_{m_k}^{n}$, $\Lambda^n_k$, $\psi^{u,n}_k$, $\psi^{d,n}_k$, $\bPP^h$)\;
            $\Lambda_k^{n+1} \gets$ SolvePrices($\phi_{\nu_k}^{n+1}$, $\phi_{m_k}^{n}$, $\Lambda^n_k$, $\psi^{u,n}_k$, $\psi^{d,n}_k$, $\bPP^h$)\;
            $\phi_{m_k}^{n+1} \gets$ SolveDriftVol($\phi_{\nu_k}^{n+1}$, $\phi_{m_k}^{n}$, $\Lambda^{n+1}_k$, $\psi^{u,n}_k$, $\psi^{d,n}_k$, $\bPP^h$)\;
            $\psi^{u,n+1}_{k+1} \gets$ UpdatePsiUp($\psi^{u,n+1}_{k+1}$, $\phi_{\nu_k}^{n+1}$, $\phi_{b_k}^{n+1}$, $\Lambda^{n+1}_k$, $\bPP^h$)\;
        }
        $\phi_{\nu_{N_T}}^{n+1} \gets$ SolveMarginal($\phi_{\nu_{N_T}}^{n}$, $\Lambda^{n}_k$, $\psi^{u,n}_{N_T}$, $\psi^{d,n}_{N_T}$, $\bPP^h$)\;
        $\Lambda_{N_T}^{n+1} \gets$ SolvePrices($\phi_{\nu_{N_T}}^{n+1}$, $\Lambda^{n}_k$, $\psi^{u,n}_{N_T}$, $\psi^{d,n}_{N_T}$, $\bPP^h$)\;
        $e_{\text{max}} \gets \frac{\|\Phi^{n+1} - \Phi^n\|_\infty}{\|\Phi^n\|_\infty}$\;
        \If{$e_{\text{max}} < \epsilon$}{
            \Return $\Phi$, $\Psi$\;
        }
    }
    \Return $\Phi$, $\Psi$\;
\end{algorithm}
The functions SolveMarginal, SolvePrices and SolveDriftVol are functions that solve the maximimization problems in \eqref{sinkhorn}.
The functions UpdatePsiUp and UpdatePsiDown correspond to the  recursive definitions in proposition \ref{ud}).

\subsection{Space  truncation and discretization}
\label{truncate}

We expect  the number of (Newton) iterates to solve the maximization problems (\ref{sinkhorn}) is finite and  the intermediate variables $\psi^d_k$ and $\psi^u_k$  (proposition 
\ref{ud}) can be stored instead of recomputed inside the sinkhorn loop. Then for all $k \in  1 \ldots N_t$,   the order of number of operations will be upper bounded by 
  the cost of computing  for all $x_k \in \Xc_k$   convolutions  in $x_{k+1} \in \Xc_{k+1}$ with the  kernel   
  \begin{equation} 
  \label{kernel} 
  \Kc(x_k,x_{k+1}) =   \exp\left( \frac{\Delta_{k,k+1}}{h}\right)  \bPP^h_{k,k+1}.
    \end{equation}    

Before discretizing in space we need to 
 truncate the support  so that it has a finite width. 
The simplest approach is to assume that the solution remains close to the reference measure and use its standard deviations $ \overline{\sigma_k}$.
for each timestep $t_k =  k\, h \quad i = 0.,\ldots,N_T$,  we restrict the computational domain to a $\delta$ multiple (usually $\delta = 5$) of the standard deviation: 
\[
    \overline{\Xc_k} = \left[ - \delta \,  v_k ,\,  \delta \, v_k  \right]
\]
where $v_k = \sqrt{v_0^2 + h \sum_{l = 0}^{k} \overline{\sigma_l}^2}$ and  $v_0$ is the standard deviation of the initial marginal $\nu_0$.  
This is also assuming the solution remains close to a martingale with $0$ mean.  There are many possible refinements, for exemple 
using a priori guess by interpolating coarse solution produced by the time multiscale strategy (see below).  

On these  compact supports, we discretize the potentials on a regular grid of size $\diff x$ and use a  parabolic scaling: for all $k$, 
\begin{equation} 
\label{PS} 
    \diff x = K \bs_k \sqrt{h}. 
\end{equation} 
 $K$  is finite (we used $K = 50$). 
  Under  this choice the number of points at time $t_k$  in space $N_{\Xc_k}$ is of order $O(\sqrt{N_T})$.
  
  We again  assume the  solution sufficiently close to a diffusion with  the  reference measure as law. 
  For all  $x_k \in \Xc_k$ , we 
    and truncate 
  the  transition kernel $\Kc(x_k,.)$   (\ref{kernel})    to $0$ outside a  sliding window $[x_k - \delta  \bs_k \sqrt{h}, x_k + \delta  \bs_k \sqrt{h} ]$.
  Therefore  $\Kc(x_k,.)$ has finite size $\delta  \, K$.   This is for a martingale process and needs of course to be adjusted using the drift if any.  
  
The number of operation for one Sinkhorn loop for this implementation is therefore of order $ O( N_T \, N_{\Xc_k}) = O( N_T^{3/2}) $.
This is confirmed experimentaly in \cite{Mpaper}.

\subsection{Acceleration of Sinkhorn iterations }
\label{sec:multiscale}

Sinkhorn iterations are known to converge slowly when the ratio between displacement and regularization parameter or  ``temperature''  goes to $0$.
In our case while the temperature scales with $h$ (and therefore goes to 0), the displacement betwen marginals of $\PP^h$ decreases.
We do not have a rigorous study of the speed of convergence for our sinkhorn iterates but we have implemented two  acceleration procedures 
that seems, at least experimentaly to improve it.

Since one Sinkhorn iteration $\Phi^{k+1} = s(\Phi^k)$ is a fixed-point iteration,  we accelerate convergence using  Anderson acceleration \cite{anderson1965} \cite{anderson2011}. \\

Because the complexity of algorithm \ref{algo:sinkhorn} scales with
the number of timesteps $N_T$,   we also  perform a   coarse to fine time discretisation  warm restart : 
One way to do so (as in  \cite{Mpaper}) consists in interpolating in time 
the coarse potentials  to initialize on a finer time grid. 
Alternately one can interpolate  the coarse solution to construct a new reference measure.  Indeed the method can be extended to non local in time and space reference 
volatilities.   The procedure reduces the number of iterations of the finer time grid and therefore computing  time.
It also seems to stabilize the optimisation as  Sinkhorn direct fine time grid arbitrary initialisation  does not  converge systematically.
Probably because the  price constraints are numerically infeasible with the initial volatility model.

\section{Application: Local volatility calibration}
\label{alsv} 

It is customary for this application to  minimize over positive martingale positive processes, hence to consider 
 processes of the form $X_t = \log S_t$ where $S_t$ is a martingale diffusion:
\begin{equation*}
    \frac{\diff S_t}{S_t} = \sigma(S_t, t) \,  {\diff \RR}_t.
\end{equation*}
Applying Ito's lemma, in terms of $X_t = \log S_t$, we obtain the following SDE:
\begin{equation*}
    \diff X_t = \mu(X_t, t) \, \diff t + \sigma(X_t, t) \, {\diff \RR}_t, \quad  \mu_t = -\frac{1}{2} \sigma_t^2
\end{equation*}

As noted in  \cite{guo2021optimal},  the process $S_t$ will be a martingale if 
\[
  b_t =   2 \mu_t + \sigma_t^2  =0 
\]
After discretisation,  this is becomes  non linear equation 
 in  the coefficients  $\beta_k$ and $\alpha_k$ defined in section \ref{sec:prelim} equations (\ref{beta_def}-\ref{alpha_def}):
\[
  b_k =   2 \beta_k + \alpha_k -  h \, \beta_k ^2  = 0  
\]
The variable $b_k$ cannot be represented as a conditional moment of $\PP$ and  we need to  back to the exponential form of the prices.
Instead of the variables $\beta$ and $\alpha$, we will use the variable $b_k$ corresponding to the choice $B(X, Y) = 1 - e^{Y - X}$:
\begin{align*}
    b_k(x) & = \frac{1}{h} \expect*{1 - e^{X_{k+1} - X_k}|X_k=x}                    \\
           & = \frac{1}{h} \frac{1}{e^{X_k}}\expect*{e^{X_k} - e^{X_{k+1}}|X_k = x} \\
           & = \frac{1}{h} \frac{1}{S_k}\expect*{S_k - S_{k+1}|X_k = x},
\end{align*}
$b_k = 0$ is the martingale constraint on the chain $S_{t_k}$.  In practice with use a soft constraint  by minimizing 
$F = c_{\text{mart}} \| b  \|_{L^2}$ with  $c_\text{mart} > 0 $ the strength of the matingale constraint penalization.  \\

For the price constraints, let $c_i \in \R^+$ be an observed price, we use the soft constraint $C_i$ a convex function with minima in $c_i$, for instance, $C_i =  \frac{1}{2}(\cdot - c_i)^2$.
We use the payoff function $G_i(x) = \max(0, e^x - K_i)$ for a call option with strike $K_i$, and $G_i(x) = \max(0, K_i - e^x)$ for a put option with strike $K_i$. 

For the initial marginal constraint, we propose using a hard constraint $M_0 = \iota_{\mu_0}$ whose dual is $\SP{\phi_{\nu_0}}{\mu_0}{}$, with $\mu_0 = \delta_{\log S_0}$. \\

Finally, we obtain the following problem:
\begin{align*}
    \mathcal{V} = \inf_{\PP^h} & \sum_0^{N_T-1}  h \expect_{\nu_k}*{F(b_k)}    
                                + M_0\left(\nu_0\right) + \sum_{i=1}^N C_i\left(g_i\right) + h \KL(\PP^h|\bPP^h) , 
\end{align*}
which  fits the framework of section \ref{ss} without reordering the calibration constraints using maturities  $(t_k)$s, 
$g_i =  \expect_{\PP^h_{\tau_i}}*{G_i({X_{\tau_i} }_\# \PP^h)}  $ is defined as in definition \ref{rem:characteristics}.

In  dual form where  (see remark \ref{R1})  we eliminate $\phi_{\nu_k} ( =  F^\star(-\phi_{m_k}) ), \, \forall k \in [1, N_T-1] $: 
\begin{align*}
    \mathcal{D} = \sup_{ \phi_{\nu_0}, \{ \phi_{m_k}\} , \{ \lambda_i \} }  & \expect_{\mu_0}*{\phi_{\nu_0} -   F^\star(-\phi_{m_0})} + \sum_{i=1}^{N_C} C_i^\star(\lambda_{i})         + h \expect_{\bPP^h}*{\exp\left(\frac{\Delta( \{  F^\star(-\phi_{m_k})  \} , \{ \phi_{m_k} \} ,\{  \lambda_i \} )}{h}\right)}. 
\end{align*}

\paragraph{Numerical results}

We construct a synthetic data set of price using  using a parametric local volatility surface.

The local volatility surface that we choose is the SSVI surface as presented in \cite{Gatheral_Jacquier_2012}.
We choose the at-the-money implied total variance for the money to be $\theta_t = 0.04t$. We choose a power-law parameterization of the function $\phi$
described in \cite{Gatheral_Jacquier_2012} as $\phi(\theta) = \eta \theta^{-\lambda}$. The at-the-money total implied variance is then
\[
    \sigma_{\text{BS}}^2(k, T) = \frac{\theta_t}{2} \left(1 + \rho \phi(\theta_t) k + \sqrt{(\phi(\theta_t) k + \rho)^2 + (1 - \rho^2)}\right),
\]
where $k$ is the log-moneyness $\log(K/F)$. The parameters are chosen as $\eta = 1.6$, $\lambda = 0.4$ and $\rho = -0.15$.
The resulting surface is shown in Figure \ref{fig:expesvi_simp}.
Prices are generated using the Black-Scholes formula. \\

We select five observation times to calibrate the model on generated prices : $t \in \{0.2, 0.4, 0.6, 0.8, 1.0\}$.
At  time $\tau_i$, we select the calls with strikes $K \in \{S_0 + 1 + 4k_i\}$ and the puts with strikes $K \in \{S_0 - 1 - 4k_i\}$ for $k_i \in \{0, 1, \dots, N_{C,i}\}$, with $N_{C} = (5, 7, 9, 10, 12)$. We calibrate
less points for the earlier maturities as mass is almost nonexistent far from the at-the-money price at these maturities.  
We choose $c_{\text{mart}} = 1\times 10^4$ as the penalization term for the martingale constraint. \\

We use the multiscale in time strategy and figure \ref{fig:expesvi_priceerr} shows the convergence of the price calibration cost versus the Sinkhorn iterations.
Each vertical bar corresponds to a refinement of the time discretisation and a reinitialisation of the reference measure.  The finest scale is 
$N_T=81$.   Figure \ref{fig:expesvi_itererr} shows the
$L^2$ norm of the relative iterate errors $\frac{\|\Phi^{k+1} - \Phi^k\|}{\|\Phi^k\|}$ versus Sinkhorm iterations at the finest scale. 
 Figure \ref{fig:expesvi_marterr} shows the $L^2$ norm of the martingale error. 
Finally, Figure \ref{fig:expe2_cal} shows the volatility  calibration results at each time. For each of the calibration times,
we show the reference implied volatility, the calibrated implied volatility, and the implied volatility generated by
the forward diffusion process with the same number of timesteps and the volatility of the solution. \\

Algorithm \ref{algo:sinkhorn} is implemented  in Python using the PyKeops library. For this simulation 
The program runs in approximately 10 minutes on a V100 GPU with 24GB of GDDR5 memory and an Intel Xeon 5217 8 core CPU with 192GB or DDR4 RAM.

\begin{figure}[ht]
    \centering
    \includegraphics[width=0.7\textwidth]{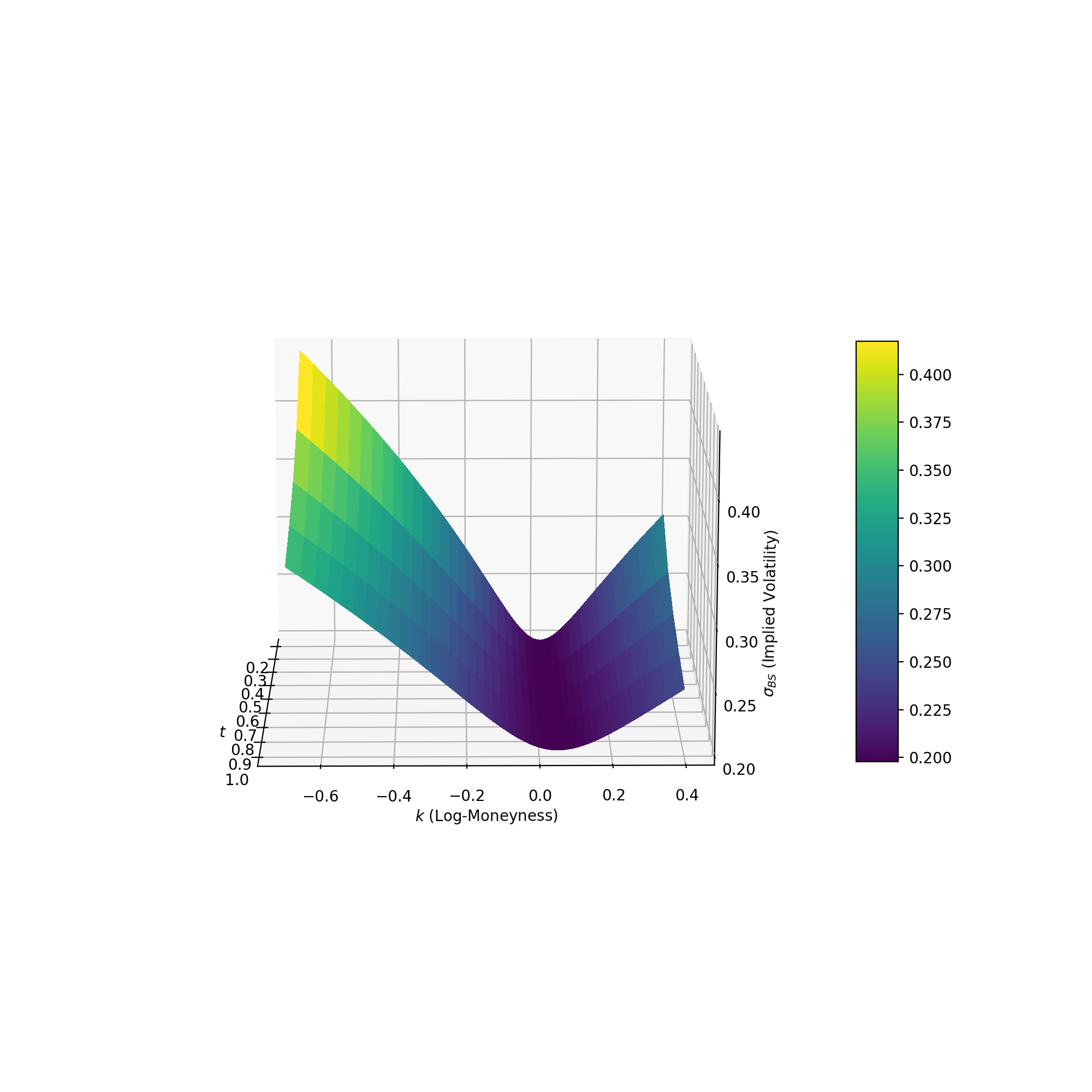}
    \caption{Generating model implied volatility}
    \label{fig:expesvi_simp}
\end{figure}

\begin{figure}[ht]
    \centering
    \begin{subfigure}[t]{0.45\textwidth}
        \includegraphics[width=\textwidth]{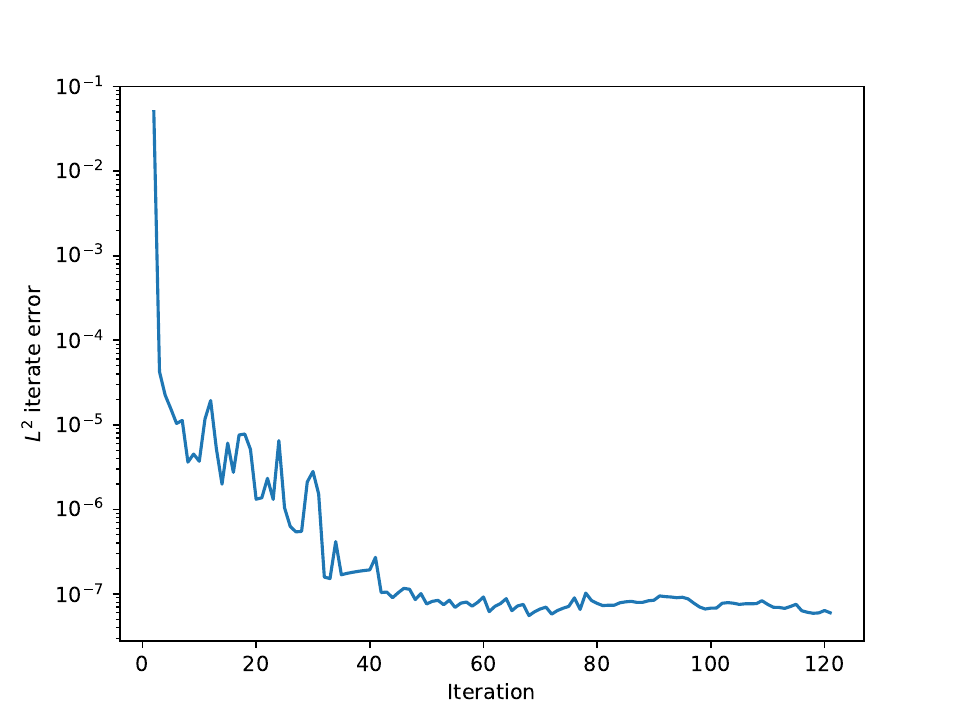}
        \caption{$L^2$ norm of iterate errors}
        \label{fig:expesvi_itererr}
    \end{subfigure}
    \begin{subfigure}[t]{0.45\textwidth}
        \includegraphics[width=\textwidth]{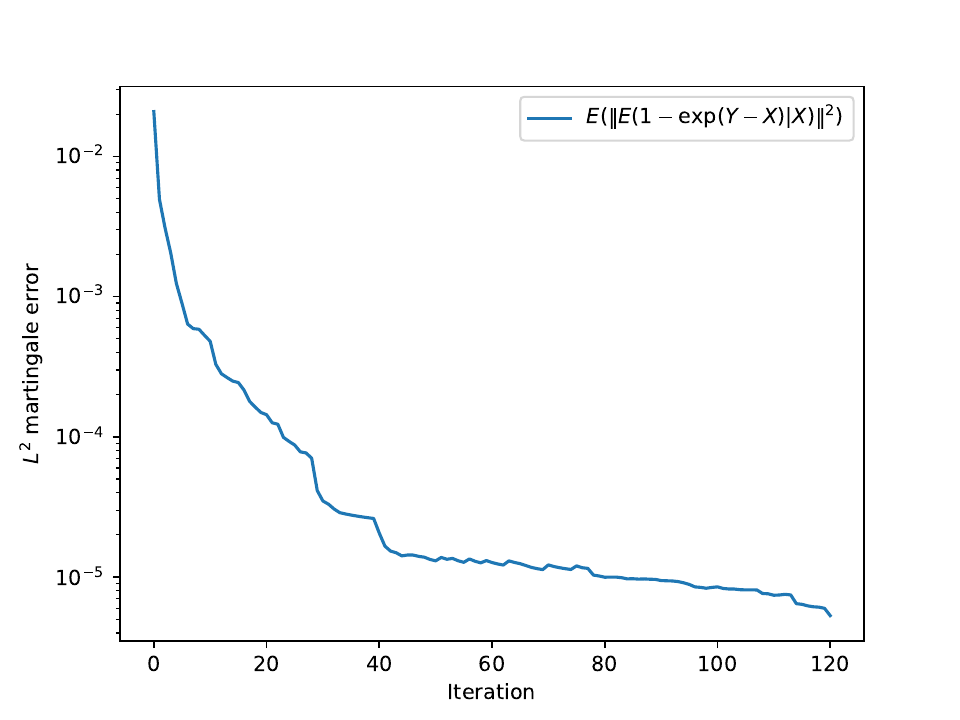}
        \caption{$L^2$ norm of the martingale error}
        \label{fig:expesvi_marterr}
    \end{subfigure}
    \begin{subfigure}[t]{0.45\textwidth}
        \includegraphics[width=\textwidth]{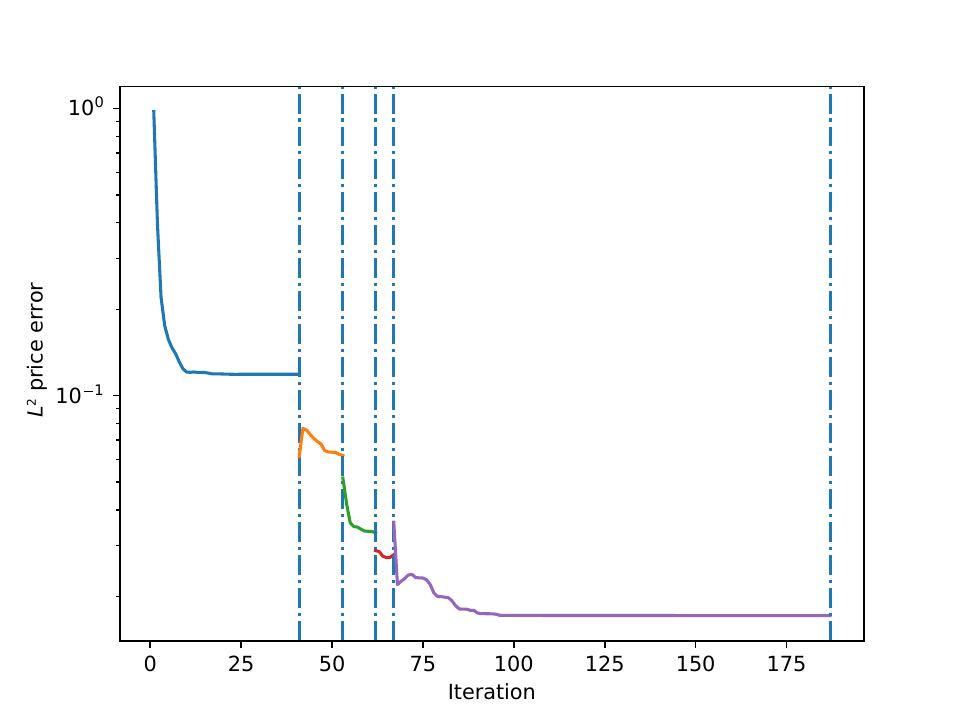}
        \caption{$L^2$ norm  of the price errors}
        \label{fig:expesvi_priceerr}
    \end{subfigure}
    \caption{Convergence curves}
    \label{fig:expe2_conv}
\end{figure}

\begin{figure}[ht]
    \centering
    \begin{subfigure}[t]{0.45\textwidth}
        \includegraphics[width=\textwidth]{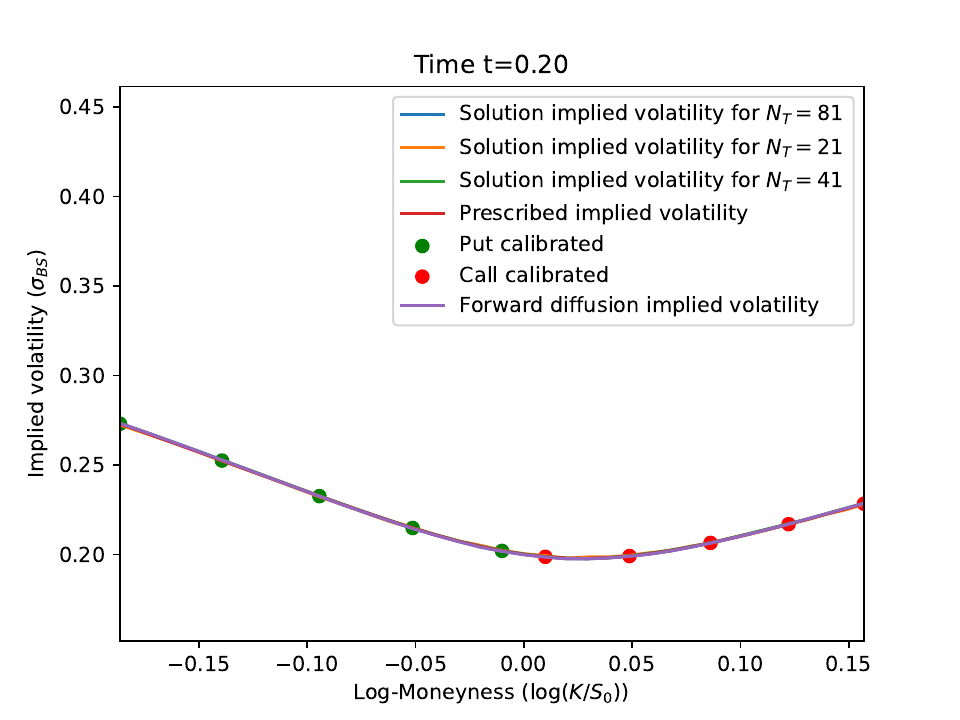}
        \caption{Calibration at time 0.2.}
        \label{fig:expesvi_cal02}
    \end{subfigure}
    \begin{subfigure}[t]{0.45\textwidth}
        \includegraphics[width=\textwidth]{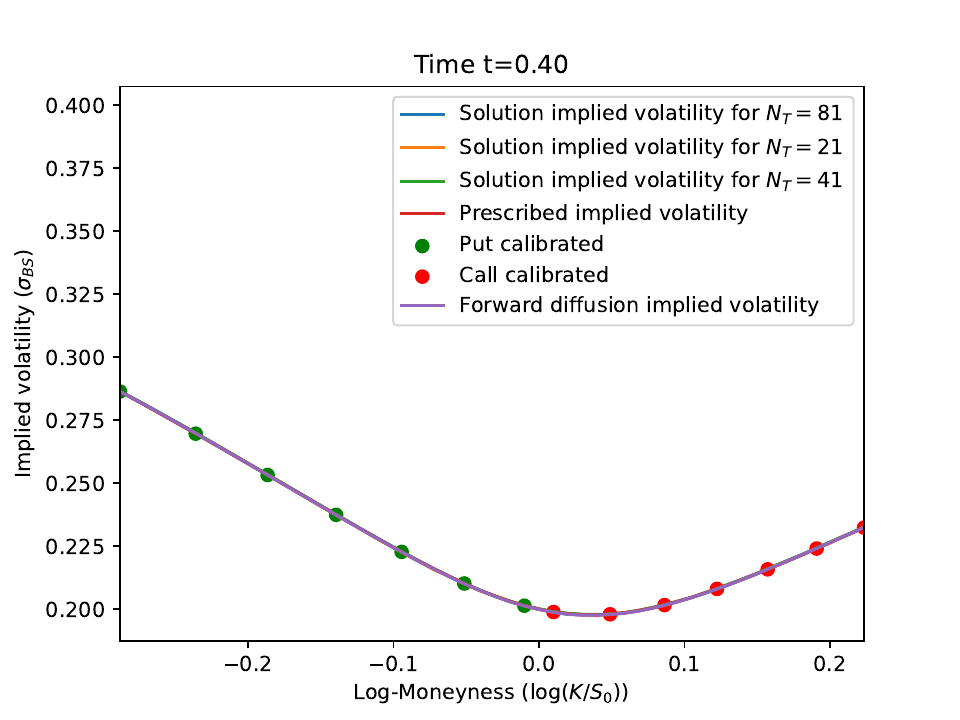}
        \caption{Calibration at time 0.4.}
        \label{fig:expesvi_cal04}
    \end{subfigure}
    \begin{subfigure}[t]{0.45\textwidth}
        \includegraphics[width=\textwidth]{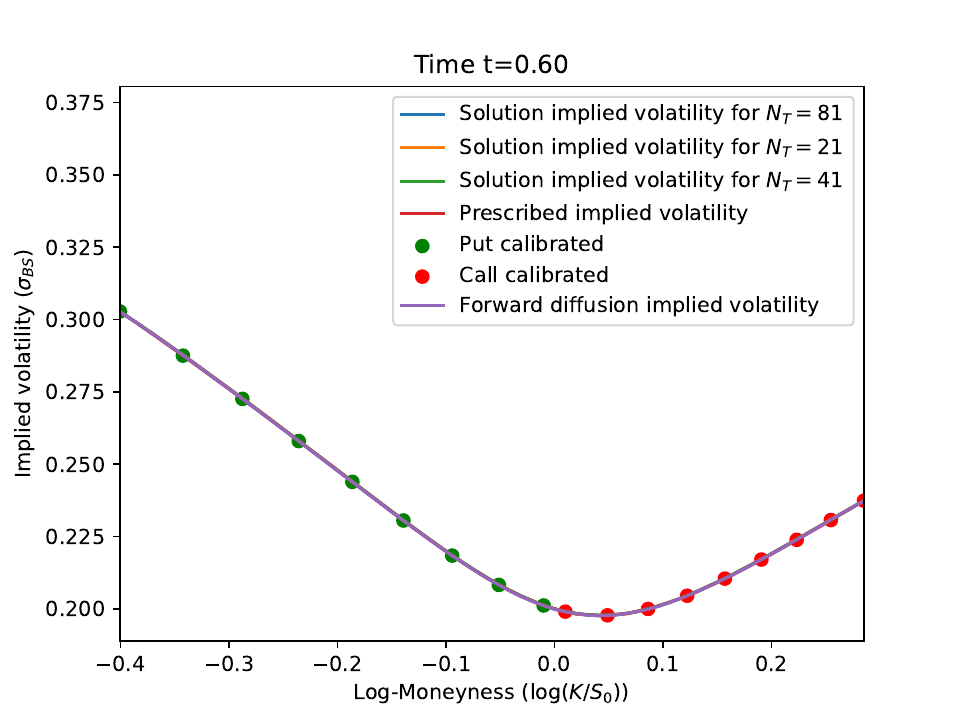}
        \caption{Calibration at time 0.6.}
        \label{fig:expesvi_cal06}
    \end{subfigure}
    \begin{subfigure}[t]{0.45\textwidth}
        \includegraphics[width=\textwidth]{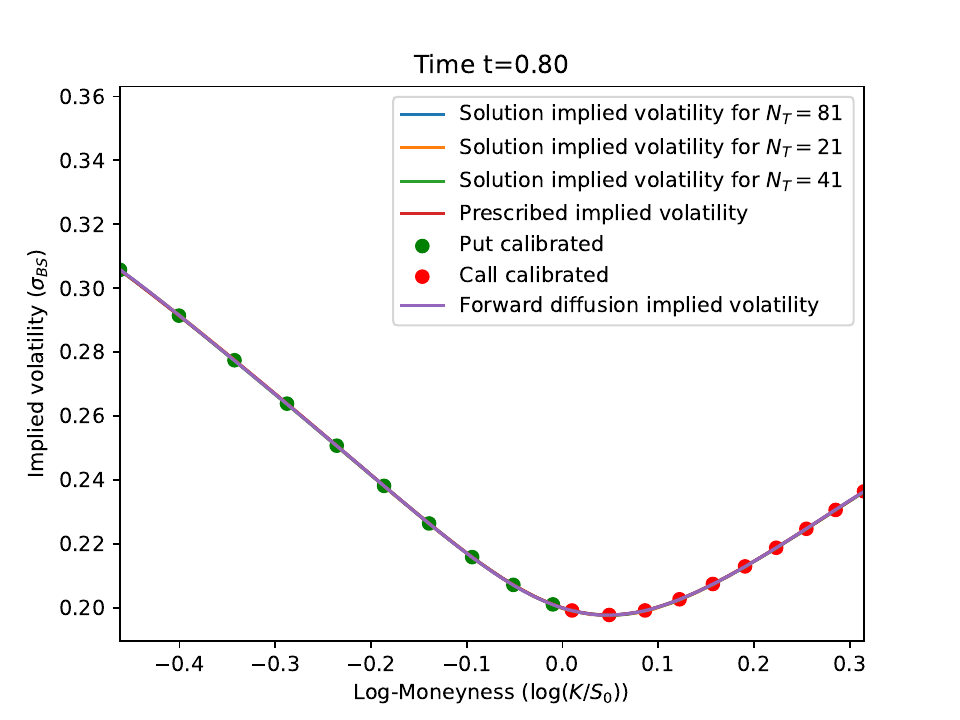}
        \caption{Calibration at time 0.8.}
        \label{fig:expesvi_cal08}
    \end{subfigure}
    \begin{subfigure}[t]{0.45\textwidth}
        \includegraphics[width=\textwidth]{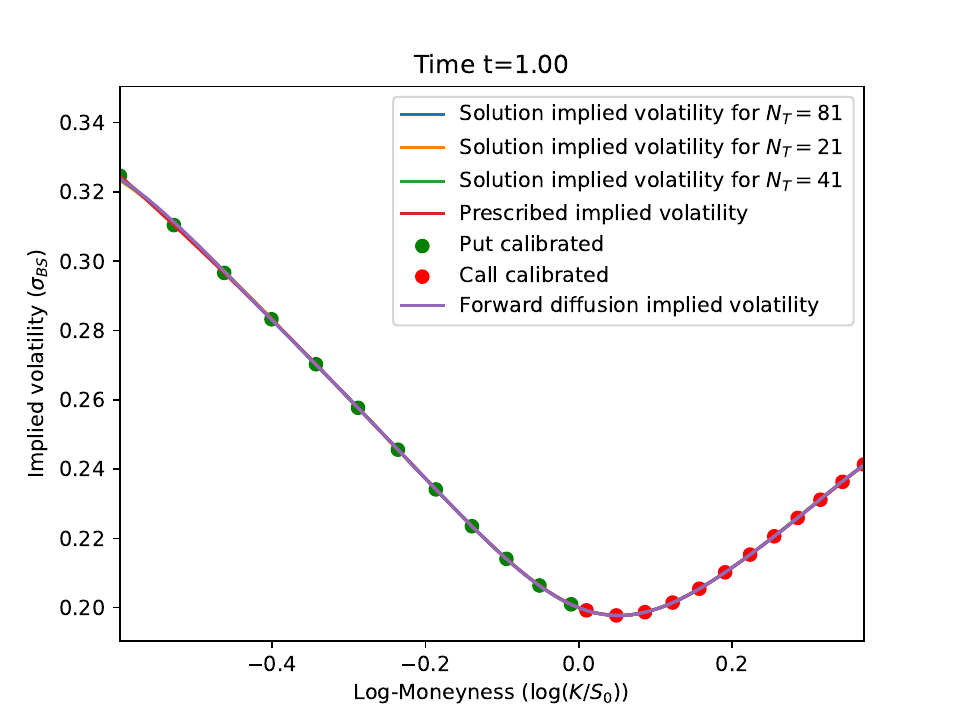}
        \caption{Calibration at time 1.}
        \label{fig:expesvi_cal10}
    \end{subfigure}
    \caption{Calibration results.}
    \label{fig:expe2_cal}
\end{figure}

\printbibliography

\appendix

\section{Fenchel Rockaffelar} 
\label{FRC} 
We  recall the abstract  Fenchel-Rockafellar theorem and hence the form of the primal problem 
\Nte{ Changer E et F } 
\begin{theorem}[Fenchel-Rockafellar]
    \label{tfr}
    Let $(E, E^*)$ and $(F, F^*)$ be two couples of topologically paired spaces.
    Let $\Delta : E \rightarrow F$ be a continuous linear operator and $\Delta^\dagger : F^* \rightarrow E^*$ be its adjoint. Let $\Fc : E^* \to \overline{\R}$ and $\Gc : F^* \to \overline{\R}$ be two lower semicontinuous and proper convex functions, 
    $\Fc^\star$ and $\Gc^\star$ their Legendre-Fenchel transform.
     If there exists $\PP \in F^*$ such that $\Gc(\PP) < +\infty$ and $\Fc$ is continuous at $\Delta^\dagger \PP$, then :
    \begin{equation*}
        \sup_{\Phi \in E} -\Fc^\star(-\Phi) - \Gc^\star(\Delta \Phi) = \inf_{\PP \in F^*} \Fc(\Delta^\dagger \PP) + \Gc(\PP),
    \end{equation*}
    and the $\inf$ is attained. Moreover, if there exists a maximizer ${\Psi} \in E$, then there exists $\PP \in F^*$ satisfying $\Delta {\Psi} \in \partial \Gc(\PP)$ and $\Delta^\dagger \PP \in -\partial \Fc^\star(- {\Psi})$.
\end{theorem}

We note the primal problem :
\begin{equation}
    \mathcal{V} := \min_{\PP \in F^*} \Fc(\Delta^\dagger \PP) + \Gc(\PP) 
\end{equation} 
and the dual :
\begin{equation}
    \mathcal{D} := \sup_{\Phi \in E} -\Fc^\star(-\Phi) - \Gc^\star(\Delta \Phi) 
\end{equation}

\section{Proof of Proposition \ref{prop_markovian}}
\label{proof_propmarkov}
First, we separate the sum of $\lambda_{g_i}$ per timesteps using the values defined above :
\begin{equation*}
    \sum_{i=0}^{N_C} \lambda_{g_i} G_i(x_i) = \sum_{k=0}^{N_T} \sum_{i=0}^{N_C} \lambda_{g_i} \mathds{1}_{\tau_i = k} G_i(x_i) = \sum_{k=0}^{N_T} \vv{\Lambda}_k \cdot \vv{G_k}(x_k)
\end{equation*}

We can rewrite the operator $\Delta$ as a sum :
\begin{align*}
    \Delta(\phi_{m}, \phi_{b}, \lambda_g) = & \sum_{k=0}^{N_T - 1} \Delta_{k,k+1}(x_k, x_{k+1}) +  \frac{1}{h}  \,\vv{\Lambda_k} \cdot \vv{G_k} + \phi_{{\nu_k}}(x_k) \\
                                            & + \phi_{m_{N_T}}(x_{N_T}) + \frac{1}{h}  \, \vv{\Lambda_{N_T}} \cdot \vv{G_{N_T}}(x_{N_T})
\end{align*}
where only consecutive timesteps are grouped together. In particular, for a given $k$, we can separate this sum into three parts :
\begin{align*}
    \Delta(\phi_{m}, \phi_{p}, \phi_{d}, \lambda_g) = & \, \Delta_{k,k+1}(x_k, x_{k+1}) + \frac{1}{h}  \,  \vv{\Lambda_k} \cdot \vv{G_k}(x_k) + \phi_{{\nu_k}}(x_k) \\
                                                      & +  \frac{1}{h}  \,  \vv{\Lambda_{k+1}} \cdot \vv{G_{k+1}}(x_{k+1}) + \phi_{{\nu_{k+1}}}(x_{k+1})             \\
                                                      & + \Delta^u_k(x_k) + \Delta^d_{k+1}(x_{k+1})
\end{align*}
where $\Delta^u_k$ and $\Delta^d_k$ are given by :
\begin{align*}
    \Delta^u_k(x_k) & = \sum_{i=0}^{k-1} \Delta_{i,i+1}(x_i, x_{i+1}) +  \frac{1}{h}  \, \vv{\Lambda_i} \cdot \vv{G_i}(x_i) + \phi_{m_i}(x_i)                    \\
    \Delta^d_k(x_k) & = \sum_{i=k}^{N_T-1} \Delta_{i,i+1}(x_i, x_{i+1}) + \frac{1}{h}  \,  \vv{\Lambda_{i+1}} \cdot \vv{G_{i+1}}(x_i) + \phi_{m_{i+1}}(x_{i+1}).
\end{align*}
We further note :
\begin{align*}
    \overline{\Delta}_{k, k+1}(x_k, x_{k+1}) = & \, \Delta_{k,k+1}(x_k, x_{k+1}) +  \frac{1}{h}  \,  \vv{\Lambda_k} \cdot \vv{G_k}(x_k) + \phi_{{\nu_k}}(x_k) \\
                                               & +  \frac{1}{h}  \,  \vv{\Lambda_{k+1}} \cdot \vv{G_{k+1}}(x_{k+1}) + \phi_{{\nu_{k+1}}}(x_{k+1})             \\
\end{align*}
for simplicity.

Given that $\bPP^h$ is separable in the same fashion, we can compute the joint probability between steps $k$ and $k+1$ as :
\begin{align*}
    \PP^{h}_{k, k+1}(x_k, x_{k+1}) & = \int \PP^{h}(dx_{[0,k-1]}, x_k, x_{k+1}, dx_{[k+2, N_T]})                                                                              \\
                                         & = \int e^{(\Delta^u_k + \overline{\Delta}_{k, k+1} + \Delta^d_{k+1})/h} \rho_0 \prod_{i=0}^{N_T} \bPP_{i, i+1}^h dx_{[0, k-1]} dx_{[k+2, N_T]} \\
                                         & = \left(\int e^{\Delta^u_k/h} \rho_0 \prod_{i=0}^{k-1} \bPP_{i, i+1}^h dx_{[0, k-1]}\right)                                                    \\
                                         & \hphantom{=} \times e^{\overline{\Delta}_{k, k+1}/h} \bPP^h_{k, k+1}                                                                           \\
                                         & \hphantom{=} \times \left(\int e^{\Delta^d_{k+1}/h} \prod_{i=k+1}^{N_T} \bPP_{i, i+1}^h dx_{[k+2, N_T]}\right)                                 \\
                                         & = \exp(\psi^u_k(x_k) + \overline{\Delta}_{k, k+1}(x_k, x_{k+1})/h + \psi^d_{k+1}(x_{k+1})) \bPP_{k, k+1}^h(x_k, x_{k+1})
\end{align*}

Similarly as in the previous proof, we can compute the marginal as:
\begin{align*}
    \PP^{h}_k(x_k) & = \int \PP^{h}(dx_{-k}, x_k)                                                                                                                                      \\
                         & = \int e^{(\Delta^u_k + \phi_{{\nu_k}}(x_k) + \vv{\Lambda_k} \cdot \vv{G_k}(x_k)  + \Delta^d_{k})/h} \rho_0 \prod_{i=0}^{N_T} \bPP_{i, i+1}^h dx_{[0, k-1]} dx_{[k+2, N_T]} \\
                         & = \left(\int e^{\Delta^u_k/h} \rho_0 \prod_{i=0}^{k-1} \bPP_{i, i+1}^h dx_{[0, k-1]}\right)                                                                             \\
                         & \hphantom{=} \times e^{\phi_{{\nu_k}}(x_k) + \vv{\Lambda_k} \cdot \vv{G_k}(x_k)/h}                                                                                          \\
                         & \hphantom{=} \times \left(\int e^{\Delta^d_{k}/h} \prod_{i=k}^{N_T} \bPP_{i, i+1}^h dx_{[k+1, N_T]}\right)                                                              \\
                         & = \exp(\psi^u_k(x_k) + \phi_{{\nu_k}}(x_k) + \frac{1}{h}  \,  \vv{\Lambda_k} \cdot \vv{G_k}(x_k) + \psi^d_k(x_k))
\end{align*}

\end{document}